\documentclass[bj]{imsart}

\RequirePackage[OT1]{fontenc}
\RequirePackage{amsthm,amsmath,amssymb,natbib}
\RequirePackage[colorlinks,citecolor=blue,urlcolor=blue]{hyperref}
\usepackage{graphicx}

\arxiv{arXiv:0000.0000}

\newcommand{\be}{\begin{equation}}
\newcommand{\ee}{\end{equation}}
\newcommand{\ba}{\begin{eqnarray}}
\newcommand{\ea}{\end{eqnarray}}
\newcommand{\bas}{\begin{eqnarray*}}
\newcommand{\eas}{\end{eqnarray*}}

\renewcommand{\Pr}{{\mbox{\rm pr}}}
\newcommand{\Ps}{{\mbox{\rm pr}}}
\newcommand{\var}{{\mbox{var}}}

\newcommand{\bPhi}{\mbox{$\Phi$}}

\newcommand{\bSigma}{\mbox{\boldmath $\Sigma$}}
\newcommand{\bsSigma}{\mbox{\scriptsize \boldmath $\Sigma$}}
\newcommand{\bLambda}{\mbox{\boldmath $\Lambda$}}
\newcommand{\bsLambda}{\mbox{\boldmath \scriptsize $\Lambda$}}
\newcommand{\bmu}{\mbox{\boldmath $\mu$}}
\newcommand{\bsmu}{\mbox{\boldmath \scriptsize $\mu$}}

\newcommand{\bU}{\mbox{\bf U}}

\newcommand{\bI}{\mbox{\bf I}}

\newcommand{\szero}{\mbox{\boldmath \scriptsize $0$}}

\allowdisplaybreaks

\startlocaldefs
\numberwithin{equation}{section}
\theoremstyle{plain}
\newtheorem{theorem}{Theorem}[section]
\theoremstyle{plain}
\newtheorem{lemma}{Lemma}[section]
\theoremstyle{remark}

\theoremstyle{remark}
\newtheorem{example}{Example}[section]
\endlocaldefs

\begin{document}

\begin{frontmatter}
\title{Asymptotic coverage probabilities of bootstrap percentile confidence intervals for constrained parameters}
\runtitle{Coverage of bootstrap confidence intervals}

\begin{aug}
\author{\fnms{Chunlin} \snm{Wang}\thanksref{a}} 
\author{\fnms{Paul} \snm{Marriott}\thanksref{b}} 
\and
\author{\fnms{Pengfei}
\snm{Li}\thanksref{b,e3}%
\ead[label=e3,mark]{pengfei.li@uwaterloo.ca}%
}

\address[a]{Department of Statistics, School of Economics and Wang Yanan Institute for Studies in Economics, Xiamen University, Xiamen, 361005, China.
}

\address[b]{Department of Statistics and Actuarial Science, University of Waterloo, Waterloo, Ontario, N2L 3G1, Canada.
\printead{e3}
}

\runauthor{C. Wang, P. Marriott and P. Li}

\affiliation{Xiamen University and University of Waterloo}

\end{aug}

\begin{abstract}
The asymptotic behaviour of
the commonly used bootstrap percentile confidence interval is investigated
when the parameters are subject to linear inequality constraints.
We concentrate on the important one- and two-sample problems with data generated from general parametric distributions in the natural exponential family.
The focus of this paper is on quantifying the coverage probabilities of the parametric bootstrap percentile confidence intervals, in particular their limiting behaviour near boundaries.
We propose a local asymptotic framework
to study this  subtle coverage behaviour. Under this framework, we discover that
when the true parameters are on, or close to, the restriction boundary,
the asymptotic coverage probabilities can always exceed the nominal level in the one-sample case;
however, they can  be, remarkably, both under and over the nominal level in the two-sample case.
Using illustrative examples, we show that the results provide theoretical justification and guidance on applying the bootstrap percentile method to  constrained inference problems.
\end{abstract}


\end{frontmatter}

\section{Introduction}
\label{sec.intro}

This paper considers situations where  parameters of interest are restricted by linear inequality constraints and, in particular,  the effect of these constraints on bootstrap percentile confidence intervals.
The nature of these problems is motivated  by real applications.
For example,
\cite{Feldman1998} studied a signal plus noise problem from high energy physics  which can be modelled by a Poisson random variable with a {\em boundary constraint} on the mean parameter space.
Also, \cite{LiTaylorNan2010} provided a multiple-sample example from a pancreatic cancer biomarker study which involves {\em ordering constraints} on the parameter space of binomial probabilities.

To illustrate the issues under consideration, consider Figure~\ref{fig1.coverage.I}(a)  which shows, for a Poisson example, the exact finite sample coverage for the rate parameter, $\lambda$, with the constraint $\lambda \geq 2$ for a nominally $90\%$ bootstrap percentile confidence interval.  We see the actual coverage makes a step change near the constraint.
Figure~\ref{fig1.coverage.I}(b)  shows, for a two-sample Bernoulli example, with two proportions $p_1$ and $p_2$, the exact finite sample coverage for parameter $p_1$,  with the constraint $p_1\leq p_2$, for a nominally 90\% bootstrap percentile confidence interval.  We see the actual coverage has a subtle trend near the constraint boundary defined by $\Delta_0:=p_2-p_1$.
This paper examines these phenomenon  by examining  the asymptotic coverage.  We also see discreteness effects on top of the coverage functions due  to the finiteness of the samples used  in  calculation.

\begin{figure} 
\centering
\includegraphics[width=30pc]{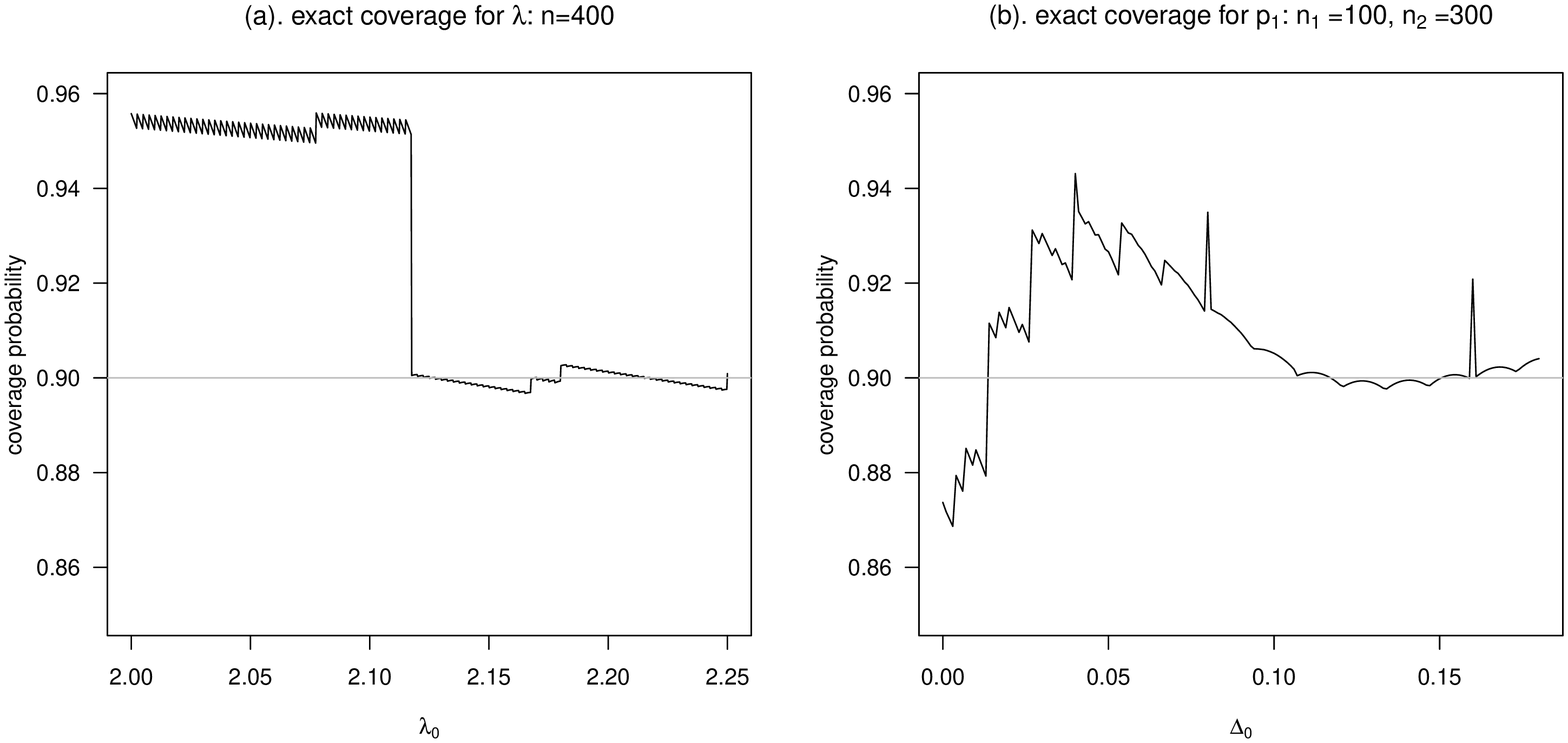}
\caption[]{Panel (a) shows exact coverage probability of bootstrap percentile confidence interval, for the  rate $\lambda$ of Poisson distribution in one-sample problem with $\lambda \geq 2$, as a function of true $\lambda_0$ and sample size $n=400$.
Panel (b) shows exact coverage probability of bootstrap percentile confidence interval, for the proportion $p_1$ of the first Bernoulli distribution in two-sample problem with $p_1\leq p_2$, as a function of true $\Delta_0=p_2-p_1$ and two sample sizes $n_1=100$ and $n_2=300$.
For both plots, we consider equal-tail intervals with 90\% nominal level. 
}
\label{fig1.coverage.I}
\end{figure}

Research on developing statistical methods and theory with such constraints has a long history; see \cite{Barlow1972}, \cite{Robertson1988}, and \cite{Silva2004}. 
The general consensus is that,
if the constraint information can be properly incorporated into the analysis,
we can expect to obtain more efficient and principled statistical inference results.

For regular parametric models, it is well-known that the likelihood ratio, score and Wald statistics
are asymptotically pivotal  under some classical regularity conditions.
This property guarantees
the  consistency of the resultant confidence intervals based on their limiting distributions.
With the inequality constraints,
the true values of the parameters may lie on the boundary of the parameter space,
which violates commonly used regularity conditions. 
As a consequence, 
these statistics 
are no longer asymptotically pivotal.
In some situations, their limiting distributions do not even have a simple analytic form.
We refer to \cite{Chernoff1954}, \cite{SelfLiang1987}, \cite{Andrews2001} and \cite{Molen2007} for general results and discussion. 


With the rapid advance in computing technology, using the bootstrap distributions  for these statistics  becomes a common alternative to construct confidence intervals. 
When the data distribution is known to follow a specific parametric model, it is natural to use the parametric bootstrap \citep{Lee1994}. 
Then a confidence interval for an unknown parameter can be constructed based on the percentiles of the bootstrap distribution of its maximum likelihood estimator.
This type of confidence interval is usually referred to as the {\em bootstrap percentile confidence interval} \citep{Efron1993}. 
It is widely used in practice because of its simplicity to implement \citep{Hall1988}. 
Throughout this paper, we always refer to the usual bootstrap in which the bootstrap sample size equals to the total sample size.

The theoretical properties of the bootstrap method when applied to regular parametric models have been well  documented; see \cite{Hall1992}, \cite{Shao1995} and \cite{Davison1997}, among others.
A natural question is how does the bootstrap method perform in the nonregular situations with inequality constrained parameter space?
Some pioneer work, including \cite{Andrews1997,Andrews2000}, showed that
the bootstrap distribution is inconsistent with the sampling distribution
of the maximum likelihood estimator of the constrained parameter when it is on the boundary. 
\cite{Drton2011} discussed the application of bootstrap method to hypothesis testing with the likelihood ratio, concluding  that its asymptotic size can be below or above the nominal level.
However, the behaviour of bootstrap confidence intervals for constrained parameters still seems unclear, especially when the true parameters
are near the restriction boundary.

The primary goal of this paper concerns   investigating the coverage probabilities of  
confidence intervals constructed by the  bootstrap percentile method using asymptotic methods.
We study the important one- and two-sample problems with data generated from general class of parametric distributions in the natural exponential family.
We evaluate the reliability of the bootstrap percentile confidence interval by answering the following specific questions:
(i) Can it achieve nominal coverage?
(ii) If not, does it over- or under-cover the true value of the parameter?
(iii) How can we more appropriately quantify the asymptotic coverage probability? 

The paper is organized as follows.
In Section \ref{sec.case1}, we investigate the one-sample problem with data generated from distributions in the natural exponential family and the mean parameter is subject to a boundary constraint.
In Section \ref{sec.case2}, we investigate the two-sample problem with data that also come from distributions in the natural exponential family and the two mean parameters are subject to an ordering constraint.
Under a local asymptotic framework, we quantify the asymptotic coverage probabilities of the bootstrap percentile confidence intervals for each of the one- and two-sample problems in Sections \ref{sec.case1} and \ref{sec.case2}, respectively.
For presentational convenience, proofs for Sections \ref{sec.case1} and \ref{sec.case2} are given in Sections \ref{sec3.one-sample} and \ref{sec4.two-sample}, respectively.

\section{One-sample natural exponential family}
\label{sec.case1}

\subsection{Problem setup}
We first consider the one-sample problem with a linear inequality constraint on the mean parameter space. 
Suppose $X_1,\ldots,X_n$ is a random sample from a general parametric distribution in the natural exponential family with probability density function  or probability mass function
\be
\label{expfamily}
f(x; \theta) = a(x)\exp\left\{ \psi x-b(\psi) \right\},
\ee
where $\psi$ is the natural parameter, and $\theta=E(X_1)=b'(\psi)$ represents the mean parameter.
It is assumed that $b(\cdot)$ is twice continuously differentiable with $b''(\psi)>0$.
Let $\sigma^2=b''(\psi)$ be the variance of $X_1$ under $f(x; \theta)$.
The parameter space of $\theta$ is constrained in $\mathcal{C}_1=\{\theta: \theta\geq d\}$ for some fixed boundary $d$.
We aim to quantify the coverage probability of the bootstrap percentile confidence interval for $\theta$. 

Based on $n$ random observations from (\ref{expfamily}), the log-likelihood function of $\theta$,
up to a constant not dependent on $\theta$,
is
$
l_n(\theta)=\psi\sum_{i=1}^nX_i-nb(\psi).
$
Then the maximum likelihood estimator of $\theta$ is defined as
$
\hat\theta_n=\arg\max_{\theta\in \mathcal{C}_1 }l_n(\theta).
$
The following lemma finds the closed form of  $\hat\theta_n$.

\begin{lemma}
\label{lem3_bootCI}
Suppose $X_1,\ldots,X_n$ is a random sample from  $f(x;\theta)$ defined in (\ref{expfamily}).
The maximum likelihood estimator of $\theta$ subject to the constraint in $\mathcal{C}_1$ is $\hat \theta_n=\max(\bar X_n , d)$ where $\bar X_n=\sum_{i=1}^n X_{i}/n$.
\end{lemma}

Next, we construct the bootstrap percentile confidence interval of $\theta$ based on $\hat\theta_n$.
Let $X_1^*,\ldots,X_n^*$
denote a parametric bootstrap sample from the model $f(x;\hat\theta_n)$, for given $\hat\theta_n$,
and let
$\hat\theta^*_n=\max(\bar X_n^{*} , d)$
denote the maximum likelihood estimator of $\theta$
based on the bootstrap sample, where $\bar X_n^*=\sum_{i=1}^nX_i^*/n$.
Further, let
$
G_{n}^*(x;\hat\theta_n)= \Ps(\hat\theta_n^*\leq x \mid \hat\theta_n)
$
be the bootstrap distribution function of ${\hat\theta_n}$,
and  $q^*_{\alpha}$
be the $\alpha$-quantile of $G_{n}^*(x;\hat\theta_n)$.
Then $[q^*_{\alpha_1},q^*_{1-\alpha_2}]$ is called a nominally $100(1-\alpha)\%$ level bootstrap percentile confidence interval of $\theta$, and its \emph{coverage probability} is defined to be
\bas
\Pr\left(\theta_0\in[q^*_{\alpha_1},q^*_{1-\alpha_2}]\right),
\eas
where $\theta_0$ is the true value of $\theta$, $\alpha=\alpha_1+\alpha_2$ with $\alpha_1,\alpha_2\in(0, 0.5)$,  and $\Pr(\cdot)$ indicates the probability under $f(x;\theta_0)$.

Under the general parametric model setup, the explicit form of the coverage probability of
the bootstrap percentile confidence interval of $\theta$ is typically unavailable, or has to be derived case by case.
Therefore, it is of interest to quantify the asymptotic coverage probability of the bootstrap percentile confidence interval for $\theta$ as $n\to\infty$.

\subsection{An asymptotic framework} 
\label{sec.asy.framework}

Our asymptotic framework is motivated by a simple, but generic, example discussed
in \cite{Andrews1997,Andrews2000}.
%

\begin{example}[Normal with nonnegative mean]
\label{one-sample-normal}
Suppose $X_1, \ldots,X_n$ is a random sample from a normal distribution with mean $\theta$ and unit variance.
The parameter space of $\theta$ is constrained to be  $\{\theta:\theta\geq0\}$. 
Define 
$y^{+}=\max(y,0)$.
From Lemma \ref{lem3_bootCI},  the maximum likelihood estimator of $\theta$
is
$
\hat\theta_n
=\bar X_n^{+}
$.
Furthermore,
$n^{1/2}(\hat\theta_n-\theta_0)$ has the same distribution as $\max( Z, - n^{1/2}\theta_0 )$, where $Z$ is a standard normal random variable with zero mean and unit variance.
%
In this special case, we are able to calculate the \emph{exact} coverage probability of a bootstrap percentile confidence interval for $\theta$ as follows: 
\ba
\label{prop1_bootCI}
\Pr\left(\theta_0\in[q^*_{\alpha_1},q^*_{1-\alpha_2}]\right)
&=&
\left\{
\begin{array}{ll}
1-\alpha_1-\alpha_2, \quad &  n^{1/2}\theta_0>\Phi^{-1}(1-\alpha_2)\\
1-\alpha_1, \quad &  n^{1/2}\theta_0\leq\Phi^{-1}(1-\alpha_2)
\end{array}
\right.,
\ea
where $\Phi(\cdot)$ is the cumulative distribution function of a standard normal random variable.
\end{example}

We comment that the exact coverage probability in (\ref{prop1_bootCI}) is a piecewise constant function of $\theta_0$ with a step change at $n^{-1/2}\Phi^{-1}(1-\alpha_2)$, for given $n$ and $\alpha_2$ values.
Also, from (\ref{prop1_bootCI}), we can see that the coverage for $\theta$
depends on how close the true value $\theta_0$ is to the boundary.
The magnitude of closeness of $\theta_0$ to the boundary crucially depends on the sample size by an order of ${n}^{1/2}$.

As an illustration, Figure~\ref{fig2.coverage.I}(a) plots the exact coverage probabilities versus the true value of $\theta_0$ for different sample sizes $n$ at level $1-\alpha=0.90$ with $\alpha_1=\alpha_2=\alpha/2$.
We see that the bootstrap percentile confidence interval may behave conservatively in terms of its coverage.
Specifically, when the true value $\theta_0$ is on, or close to, the boundary, over-coverage can happen.
To further verify the trend that we have seen for the normal example, recall that, in Figure~\ref{fig1.coverage.I}(a), we also plot the exact coverage probability of the bootstrap percentile confidence interval for the rate parameter $\lambda$ of Poisson distribution with $\lambda\geq 2$ at level $1-\alpha=0.90$ with $\alpha_1=\alpha_2=\alpha/2$ for  $n=400$.
Due to the discrete nature of the Poisson distribution, we can not always expect the exact coverage probability to achieve the nominal level.
Hence, we would hope to more appropriately quantify the observed exact coverage phenomenon in an asymptotically meaningful way.
Motivated by above discussions, 
we  adopt  a local asymptotic framework
by allowing the true constrained parameter to vary in a $n^{-1/2}$-neighbourhood of the boundary.
More precisely, we let $\theta_0=\theta_{0,n}=d+\tau n^{-1/2}$.
The corresponding local parameter $\tau=n^{1/2}(\theta_{0,n} - d)$ controls the order of closeness of $\theta_{0,n}$ approaching to the boundary $d$.
This framework helps capture the subtle asymptotic coverage behaviour of the bootstrap percentile confidence interval in terms of $\tau$ for the general distributions in the natural exponential family with a constrained parameter space. 

\begin{figure} 
\centering
\includegraphics[width=30pc]{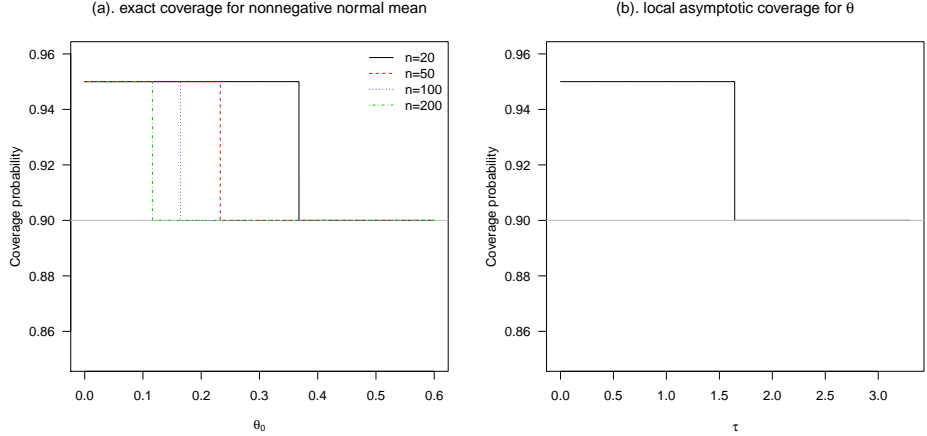}
\caption[]{Panel (a) shows exact coverage probability of bootstrap percentile confidence interval, for the  nonnegative mean parameter of normal distribution with unit variance, as a function of true $\theta_0$. The lines are for $n=20$ (solid), $50$ (dash), $100$ (dot) and $200$ (dot-dash).
Panel (b) shows quantified local asymptotic coverage probability of bootstrap percentile confidence interval, for the boundary constrained $\theta$ of general distributions in the one-sample natural exponential family with $\sigma_0=1$, as a function of $\tau$.
For both plots, we set $1-\alpha=0.90$ with $\alpha_1= \alpha_2=0.05$.}
\label{fig2.coverage.I}
\end{figure}

\subsection{General results}
\label{sec2.main}

Under the local asymptotic framework proposed in Section \ref{sec.asy.framework}, we first study the asymptotic distribution of the maximum likelihood estimator $\hat \theta_n$.

\begin{lemma}
\label{lem4_bootCI}
Suppose $X_1,\ldots,X_n$ is a random sample from $f(x;\theta)$ defined in (\ref{expfamily}),
and the true value of $\theta$ is  $\theta_{0,n}=d+\tau n^{-1/2}$ with $\tau$ being a fixed nonnegative local parameter not depending on $n$. Let $\sigma_0^2=b''(\psi_0)$ with $\psi_0=b'^{-1}(d)$.
Then, as $n\to\infty$, we have
\bas
{ n^{1/2}(\hat\theta_n-\theta_{0,n})}/{\sigma_0}
&\to&
\max\left( Z, -{\tau}/{\sigma_0} \right),
\eas
in distribution, where $Z$ is a standard normal random variable.
\end{lemma}

Using Lemmas \ref{lem3_bootCI} and \ref{lem4_bootCI},
we quantify the local asymptotic coverage probability of bootstrap percentile confidence interval for $\theta$ in the following theorem.

%
%

\begin{theorem}
\label{prop3_bootCI}
Under the same setup and assumptions as in Lemma \ref{lem4_bootCI}, as $n\rightarrow\infty$, we have
\bas
\Pr\left(\theta_{0,n}\in[q^*_{\alpha_1},q^*_{1-\alpha_2}]\right)
&\to&
\left\{
\begin{array}{ll}
1-\alpha_1-\alpha_2,\quad& \tau >\Phi^{-1}(1-\alpha_2)\sigma_0\\
1-\alpha_1,\quad& \tau < \Phi^{-1}(1-\alpha_2)\sigma_0
\end{array}
\right..
\eas
\end{theorem}

This asymptotic result generalizes the exact finite sample result in Example \ref{one-sample-normal} for the normal distribution to cover the natural exponential family of distributions.
The result is plotted in Figure~\ref{fig2.coverage.I}(b) versus the local parameter $\tau$ at level $1-\alpha=0.90$ with $\alpha_1=\alpha_2=0.05$ and  $\sigma_0=1$.
We note that  different choices of $\alpha_1$ and $\alpha_2$ should give different graphs.
If such an approximation remains good for finite $n$, then this local asymptotic result provides us with information on how likely we are to have conservative conclusions, when $\theta_{0,n}$ is shrinking close to $d$.

%

\subsection{Illustrative examples}
\label{sec2.examples}

In Section \ref{sec2.main}, we quantified the local asymptotic coverage probabilities of bootstrap percentile confidence intervals for boundary constrained parameters when the data come from  the general class of distributions in the natural exponential family.
In many applications, as seen in the introduction, observations may come from specific parametric distributions,  such as Poisson or binomial, in which their mean parameters are subject to linear inequality constraints.
To make the asymptotic result of Theorem \ref{prop3_bootCI} practically useful, we illustrate with two examples in this section.
In these examples, we compare the \emph{exact} coverage probabilities calculated under specific parametric models and the quantified \emph{asymptotic} coverage probabilities.

\begin{example}[One-sample Poisson example] 
\label{case3.pois}
Suppose $X_{1},\ldots,X_{n}$ is a random sample from a $Poisson(\lambda)$ distribution with mean $\lambda$ subject to constraint $\lambda\in[d, \infty)$ with $d>0$.
For illustration, we consider $d=2$.
Under this setup, $\sigma_0^2 =d$.

Suppose the true value of $\lambda$ is $\lambda_{0,n}=2+ \tau n^{-1/2}$.
Applying Theorem \ref{prop3_bootCI},
the local asymptotic coverage probability of the $100(1-\alpha)\%$ bootstrap percentile confidence interval of $\lambda$ is
\bas
\lim_{n\to\infty}
\Pr\left(\lambda_{0,n}\in[q^*_{\alpha_1},q^*_{1-\alpha_2}]\right)
&=&
\left\{
\begin{array}{ll}
1-\alpha_1-\alpha_2, \quad &  \tau > \Phi^{-1}(1-\alpha_2)2^{1/2} \\
1-\alpha_1, \quad &  \tau< \Phi^{-1}(1-\alpha_2)2^{1/2} \end{array}
\right..
\eas

The exact coverage probability of the $100(1-\alpha)\%$ bootstrap percentile confidence interval for $\lambda$ can also be calculated by noting the fact that a sum of independent Poisson random variables  still has a Poisson  distribution.
In Figure~\ref{fig3.coverage.pois},  we plot the asymptotic and exact coverage probabilities as functions of the true parameter $\lambda_0=\lambda_{0,n}$.
For  comparison, we also add the exact coverage probabilities of the bootstrap percentile confidence interval for $\lambda$ without using constraint.

\begin{figure} 
\centering
\includegraphics[width=30pc]{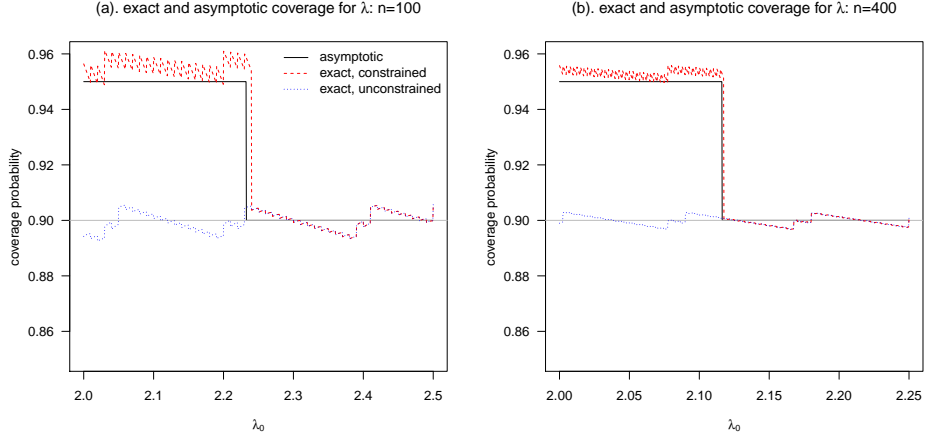}
\caption[]{Exact and asymptotic coverage probabilities of the bootstrap percentile confidence interval, for the rate $\lambda$ of Poisson distribution with $\lambda \geq 2$, as a function of $\lambda_0$.
The lines are for quantified local asymptotic coverage probability (solid), exact finite sample coverage probability (dash) and exact finite sample coverage probability without using constraint (dot).
Panel (a) is for $n=100$, and  Panel (b) is for $n=400$.
For both plots, we set $1-\alpha=0.90$ with $\alpha_1= \alpha_2=0.05$.
}
\label{fig3.coverage.pois}
\end{figure}

In Figure~\ref{fig3.coverage.pois}, we observe chaotic behaviour with oscillation phenomenon of the coverage probabilities due to the discrete nature of the Poisson distribution.
Hence, we can not always expect these coverage probabilities to achieve the nominal level due to the discrete nature of the Poisson distribution.
In general, we can see a clear trend that the quantified local asymptotic coverage probability shows close agreement with the exact finite sample coverage probability, as functions of $\lambda_0$, especially when the sample size increases.
\end{example}

\begin{example}[One-sample binomial example] 
\label{one-sample-binom}
Suppose $X_{1},\ldots,X_{n}$ is a random sample from a $Binomial(m,p)$ distribution with known $m$ and the proportion $p$ subject to constraint $ p\in[d, 1)$ with $0<d<1$.
For illustration, we consider $m=1$ and $d=0.5$.
Under the this setup, $\sigma_0^2=d(1-d)=0.25$.

Suppose the true value of $p$ is $p_{0,n}=0.5+ \tau n^{-1/2}$.
Applying Theorem \ref{prop3_bootCI}, the local asymptotic coverage probability of the $100(1-\alpha)\%$ bootstrap percentile confidence interval of $p$ is
\bas
\lim_{n\to\infty}
\Pr\left(p_{0,n}\in[q^*_{\alpha_1},q^*_{1-\alpha_2}]\right)
&=&
\left\{
\begin{array}{ll}
1-\alpha_1-\alpha_2, \quad &  \tau > 0.5\Phi^{-1}(1-\alpha_2) \\
1-\alpha_1, \quad &  \tau< 0.5\Phi^{-1}(1-\alpha_2)
\end{array}
\right..
\eas

The exact coverage probability of the $100(1-\alpha)\%$ bootstrap percentile confidence interval for $p$ can also be calculated by noting the fact that a sum of independent binomial random variables still has a binomial distribution.
Combining these results, in Figure~\ref{fig4.coverage.binom}, we plot the asymptotic and exact coverage probabilities of the bootstrap percentile confidence interval as functions of $p_0=p_{0,n}$.
For comparison, we also add the exact coverage probability of the bootstrap percentile confidence interval for $p$ without using constraint.

\begin{figure} 
\centering
\includegraphics[width=30pc]{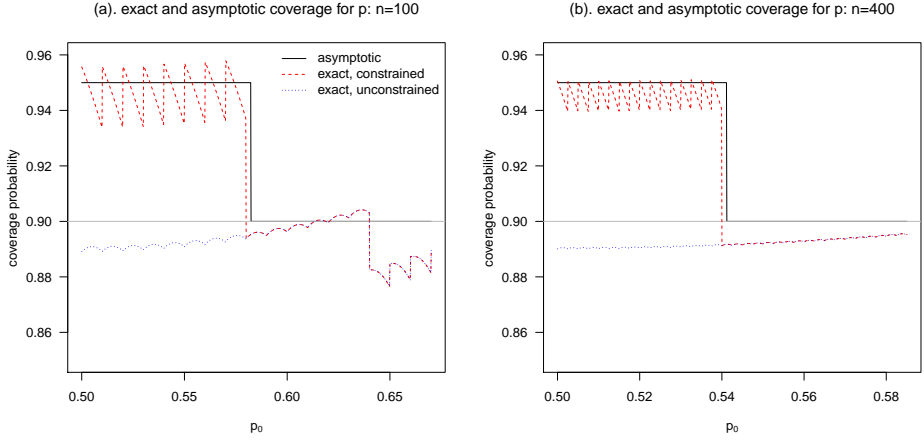}
\caption[]{Exact and asymptotic coverage probabilities of the bootstrap percentile confidence interval, for the proportion $p$ of binomial distribution with $p \in[0.5,1)$, as a function of $p_0$.
The lines are for quantified local asymptotic coverage probability (solid), exact finite sample coverage probability (dash) and exact finite sample coverage probability without using constraint (dot).
Panel (a) is for $n=100$, and  Panel (b) is for $n=400$.
For both plots, we set $1-\alpha=0.90$ with $\alpha_1= \alpha_2=0.05$.
}
\label{fig4.coverage.binom}
\end{figure}

Again, due to the discrete nature of the binomial distribution, we can not always expect the quantified coverage probabilities to achieve the nominal level, even for the exact unconstrained case.
In general, from Figure~\ref{fig4.coverage.binom}, we can observe a close agreement between the quantified asymptotic local coverage probability and the exact coverage probability, as functions of $p_0$, especially when the sample size increases.
\end{example}

\section{Two-sample natural exponential family}
\label{sec.case2}

\subsection{Problem setup}

In this section, we consider the  two-sample problem when data come from general distributions in the natural exponential family with their means subject to a linear ordering constraint.
Suppose we have random observations $X_{11},\ldots,X_{1n_1}$ from $f(x;\theta_1)$,
and independently, we have random observations $X_{21},\ldots,X_{2n_2}$ from
$f(x;\theta_2)$,
where  $f(x;\theta_i)$ satisfies (\ref{expfamily})
and $\theta_i$ still represents the mean parameter, $i=1,2$.
The parameter space of $(\theta_1,\theta_2)$ is defined to be $\mathcal{C}_2=\{(\theta_1,\theta_2):\theta_1 \leq \theta_2\}$.
Our goal is to quantify the asymptotic coverage probabilities of bootstrap percentile confidence intervals for $\theta_1$, $\theta_2$, as well as their difference $\Delta:=\theta_2-\theta_1$.

As a first step, we identify the form of the maximum likelihood estimator of $(\theta_1,\theta_2,\Delta)$.
For an asymptotic analysis, we let $\omega=n_1/n$ with $n=n_1+n_2$ and assume that $\omega\in(0,1)$ does not depend on $n$.
%
%
%
%
%
The following lemma finds the explicit forms of the  maximum likelihood estimators of $\theta_1$, $\theta_2$ and $\Delta$.
\begin{lemma}
\label{lem7_boot}
Suppose $X_{11},\ldots,X_{1n_1}$ is a random sample from $f(x;\theta_1)$,
and independently, $X_{21},\ldots,X_{2n_2}$ is another random sample from $f(x;\theta_2)$,
with $f(x;\theta_i)$ defined in (\ref{expfamily}), $i=1,2$.
Define $\bar X_{ni}=\sum_{j=1}^{n_i}X_{ij}/n_i$, $i=1,2$.
Then,  subject to the constraint in $\mathcal{C}_2$, the maximum likelihood estimators of $\theta_1$ and $\theta_2$
are
$$
\hat\theta_{n1}=\min\left\{\bar X_{n1},\omega \bar X_{n1}+(1-\omega)\bar X_{n2} \right\},
\quad
\hat\theta_{n2}=\max\left\{\bar X_{n2},\omega \bar X_{n1}+(1-\omega)\bar X_{n2} \right\},
$$
and the maximum likelihood estimator of $\Delta$ is $\hat\Delta_n=(\bar X_{n2}-\bar X_{n1})^{+}$.
\end{lemma}

%

Similarly to the discussion in Section \ref{sec.asy.framework},
fixing $(\theta_{10},\theta_{20})$ and $\Delta_0$ pointwisely
may not be helpful to reveal the subtle asymptotic coverage behaviour of the bootstrap percentile confidence intervals.
Instead we proceed by considering the following local asymptotic framework:
$$
\theta_{10}=\theta_{10,n}=\eta_0-(1-\omega)\Delta_{0,n},
\quad
\theta_{20}=\theta_{20,n}=\eta_0+\omega\Delta_{0,n},
$$
where $\Delta_{0,n}=\delta  n^{-1/2}$.
Here $\eta_0$ is a fixed value, and $\delta$ is  a fixed, nonnegative, local parameter not depending on $n$.
Under this framework, we fix the true value of $\omega\theta_{10,n}+(1-\omega)\theta_{20,n}$, i.e. the overall mean of two samples,
to be $\eta_0$, and allow the true value of constrained mean difference, $\Delta_{0,n}=\theta_{20,n}-\theta_{10,n}$,
to vary in a $n^{-1/2}$-neighbourhood of 0.

\subsection{General results}

Under this proposed local asymptotic framework, we next study the asymptotic distributions of the maximum likelihood estimators $\hat\theta_{n1}$, $\hat\theta_{n2}$ and $\hat\Delta_{n}$ in the following lemma.

\begin{lemma}
\label{lem8_boot}
Consider the same setup  and assumptions  as in Lemma \ref{lem7_boot}.
Assume the true value of $(\theta_1,\theta_2)$ is $(\theta_{10,n},\theta_{20,n})\in \mathcal{C}_2$ with $\theta_{10,n}=\eta_0-(1-\omega)\Delta_{0,n}$ and $\theta_{20,n}=\eta_0+\omega\Delta_{0,n}$,
where  $\Delta_{0,n}=\delta  n^{-1/2}$, $\eta_0$ is a fixed parameter, and $\delta$ is a fixed nonnegative local parameter not depending on $n$.
Let $\sigma_{0}^2=b''(\psi_{0})$ with $\psi_{0}=b'^{-1}(\eta_0)$.
Then, as $n\to\infty$, we have
\bas
{n^{1/2} (\hat\theta_{n1}-\theta_{10,n})}/{\sigma_{0}}
&\to&
\min\left\{ \omega^{-1/2}Z_1, \omega^{1/2}Z_1+(1-\omega)^{1/2}Z_2 + {(1-\omega)\delta}/{\sigma_0} \right\} ,  \\
{n^{1/2}(\hat\theta_{n2}-\theta_{20,n})}/{\sigma_{0}}
&\to&
\max\left\{ (1-\omega)^{-1/2}Z_2, \omega^{1/2}Z_1+(1-\omega)^{1/2}Z_2 - {\omega \delta}/{\sigma_0} \right\} , \\
{n^{1/2} (\hat\Delta_n-\Delta_{0,n})}/{\sigma_{0}}
&\to&
\max\left\{ (1-\omega)^{-1/2}Z_2-\omega^{-1/2}Z_1, -{\delta}/{\sigma_0} \right\}, 
\eas
in distribution, where $Z_1, Z_2$ are two independent standard normal random variables.
\end{lemma}

%
%

Let $X_{ij}^*$'s be the bootstrap sample
such that $X_{11}^*,\ldots,X_{1n_1}^*$ is a random sample from $f(x;\hat\theta_{n1})$ for given $\hat\theta_{n1}$,
and independently, $X_{21}^*,\ldots,X_{2n_2}^*$ is a random sample from $f(x;\hat\theta_{n2})$ for given $\hat\theta_{n2}$.
Define $\bar X_{ni}^*=\sum_{j=1}^{n_i}X_{ij}^*/n_i$, $i=1,2$.
Further, let
\bas
\hat\theta_{n1}^*=\min\{\bar X_{n1}^*,\omega \bar X_{n1}^*+(1-\omega)\bar X_{n2}^*\},
\quad
\hat\theta_{n2}^*=\max\{\bar X_{n2}^*,\omega \bar X_{n1}^*+(1-\omega)\bar X_{n2}^*\}, 
\eas
and $\hat\Delta_n^*=(\bar X_{n2}^*-\bar X_{n1}^*)^{+}$ be the maximum likelihood estimators of $\theta_1$, $\theta_2$ and $\Delta$, respectively, based on the bootstrap sample.
Denote the bootstrap distributions of $\hat\theta_{n1}$,
$\hat\theta_{n2}$ and $\hat\Delta_{n}$, respectively, by
\bas
G_{n1}^*(x;\hat\theta_{n1},\hat\theta_{n2})&=&\Ps(\hat\theta_{n1}^*\leq x\mid\hat\theta_{n1},\hat\theta_{n2} ), \\
G_{n2}^*(x;\hat\theta_{n1},\hat\theta_{n2})&=&\Ps(\hat\theta_{n2}^*\leq x\mid\hat\theta_{n1},\hat\theta_{n2} ),\\
G_{n,\Delta}^*(x;\hat\theta_{n1},\hat\theta_{n2})&=&\Ps(\hat\Delta_{n}^*\leq x\mid\hat\theta_{n1},\hat\theta_{n2} ),
\eas
and their $\alpha$-quantiles by $q_{1,\alpha}^*$, $q_{2,\alpha}^*$
and $q_{\Delta,\alpha}^*$.
%

Further, let $\Phi_{(\bsmu,\bsSigma)}(x,y)$
denote the joint cumulative distribution function
of a bivariate normal random vector with mean vector $\bmu$
and variance-covariance matrix $\bSigma$, and
let $F_{12}(x,y)$ denote the joint cumulative distribution function
of
$$
\min\left\{ \omega^{-1/2}Z_1, \omega^{1/2}Z_1+(1-\omega)^{1/2}Z_2 + {(1-\omega)\delta}/{\sigma_0} \right\} ,
$$
and
$$\max\left\{ (1-\omega)^{-1/2}Z_2, \omega^{1/2}Z_1+(1-\omega)^{1/2}Z_2 - {\omega \delta}/{\sigma_0} \right\} .
$$
That is, based on Lemma \ref{lem8_boot},
$F_{12}(x,y)$ is the joint limiting distribution of
of
$
n^{1/2}(\hat\theta_{n1}-\theta_{10,n})/\sigma_0
$
and
$
n^{1/2}(\hat\theta_{n2}-\theta_{20,n})/\sigma_0.
$
Also, define matrices
\bas
\bLambda_1=
\left( \begin{array}{cc}
 1 & \omega^{1/2} \\
 \omega^{1/2} & 1 \\
\end{array}\right),
\quad
\bLambda_2=
\left( \begin{array}{cc}
 1 & (1-\omega)^{1/2} \\
 (1-\omega)^{1/2} & 1 \\
\end{array}\right).
\eas

In the next theorem, we quantify the local asymptotic coverage probabilities of the bootstrap
percentile confidence intervals for $\theta_1$, $\theta_2$ and $\Delta$.

\begin{theorem}
\label{prop4_bootCI}
Under the same setup and assumptions as in Lemma \ref{lem8_boot}, as $n\to\infty$,
we have \\
(a)
\bas
\Pr\left(\theta_{10,n}\in[q^*_{1,\alpha_1},q^*_{1,1-\alpha_2}]\right)
&\to&
\iint  I\{\alpha_1\leq g_1(x,y)\leq1-\alpha_2\} dF_{12}(x,y) ,
\eas
where
$
g_1(x,y)
=
\Phi\left\{-C_{11}(x)\right\} + \Phi\left\{-C_{12}(x,y) \right\} - \bPhi_{(\szero,\bsLambda_1)}\left\{-C_{11}(x), -C_{12}(x,y) \right\}
$
with
$C_{11}(x)=\omega^{1/2}x$ and
$C_{12}(x,y)=\omega x+ (1-\omega)y+ (1-\omega)\delta/\sigma_0$; \\
(b)
\bas
\Pr\left(\theta_{20,n}\in[q^*_{2,\alpha_1},q^*_{2,1-\alpha_2}]\right)
&\to&
\iint I\{\alpha_1\leq g_2(x,y)\leq1-\alpha_2\} dF_{12}(x,y) ,
\eas
where
$
g_2(x,y)
= \bPhi_{(\szero,\bsLambda_2)}\left\{-C_{21}(y),  -C_{22}(x,y) \right\}
$
with
$C_{21}(y)=(1-\omega)^{1/2}y$ and
$C_{22}(x,y)=\omega x+ (1-\omega)y- \omega\delta/\sigma_0$; \\
(c)
\bas
\Pr\left(\Delta_{0,n}\in[q^*_{\Delta,\alpha_1},q^*_{\Delta,1-\alpha_2}]\right)
&\to&
\left\{
\begin{array}{ll}
1-\alpha_1-\alpha_2,\quad&  \delta > \Phi^{-1}(1-\alpha_2)\sigma_0/\{\omega(1-\omega)\}^{1/2}\\
1-\alpha_1,\quad&  \delta < \Phi^{-1}(1-\alpha_2)\sigma_0/\{\omega(1-\omega)\}^{1/2}
\end{array}
\right..
\eas
\end{theorem}

It appears that explicit expressions for the local asymptotic coverage probabilities of $\theta_1$ and $\theta_2$ are not available.
Fortunately, the quantified coverage probabilities of $\theta_1$ and $\theta_2$ are written in terms of bivariate integrals that can be easily evaluated using numerical methods.
Moreover, we note that the asymptotic coverage results only depend on the ratio $\omega$ and the local parameter $\delta$ which actually controls the true mean difference $\Delta_{0,n}$.
Figure \ref{fig5.coverage.II} plots the asymptotic coverage probabilities for $\theta_1$, $\theta_2$ and $\Delta$ versus the local parameter $\delta$, in the case of $\omega=0.1, 0.2, 0.3$ and $0.5$ at level $1-\alpha=0.90$ with $\alpha_1=\alpha_2=0.05$ as an illustration.
Again, we note that  different choices of $\alpha_1$ and $\alpha_2$ can lead to different coverage behaviours of the confidence intervals, and hence give different graphs.

\begin{figure} 
\centering
\includegraphics[width=32pc]{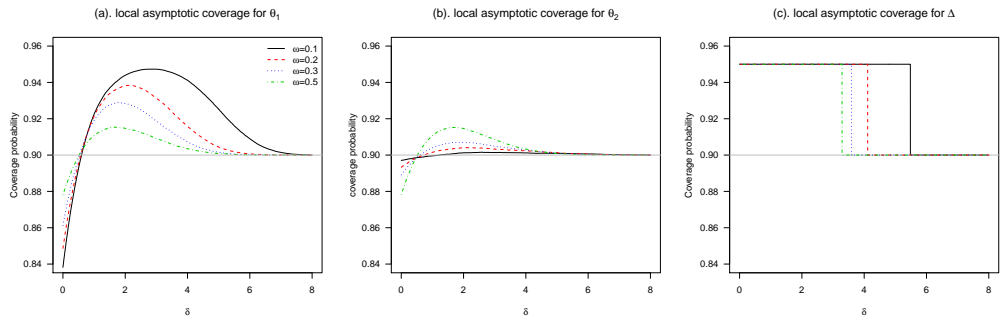}
\caption[]{
Local asymptotic coverage probabilities of bootstrap percentile confidence intervals, for $\theta_1$ in Panel (a), $\theta_2$ in Panel (b) and $\Delta=\theta_2-\theta_1$ in Panel (c) of distributions in the two-sample natural exponential family with ordering constrained  means and $\sigma_0=1$, as functions of $\delta$.
The lines are for $\omega=0.1$ (solid), $0.2$ (dash), $0.3$ (dot) and $0.5$ (dot-dash).
For all plots, we set $1-\alpha=0.90$ with $\alpha_1= \alpha_2=0.05$.
}
\label{fig5.coverage.II}
\end{figure}

%

From Figure~\ref{fig5.coverage.II}, we observe that the asymptotic coverage probabilities for $\theta_1$ and $\theta_2$ can be, somewhat surprisingly, both greater and smaller than the nominal level.
These asymptotic results show that the bootstrap percentile confidence intervals of $\theta_1$ and $\theta_2$ are no longer always conservative, but can significantly under-cover the corresponding true parameters when the constrained parameter is on, or close to, the restriction boundary. On the other hand,  when the two-sample mean difference $\Delta$ is the parameter of interest, the bootstrap percentile confidence interval of $\Delta$ still performs conservatively. 

\subsection{An illustrative example}
\label{sec3.examples}

To make our asymptotic results practically appealing,
we also apply the results of Theorem \ref{prop4_bootCI}
to a two-sample binomial example with ordered proportions. 

\begin{example} [Two-sample binomial example]
\label{case4.bino}
Suppose we have two independent samples $X_{ij} \sim Binomial(m_i,p_i)$ for $i=1,2$ and $j=i,\ldots,n_i$, with known $m_i$, where $p_1$, $p_2$ are subject to the constraint $p_2 \geq p_1$.
For illustration, we consider $m_1=m_2=1$.
Let $\Delta=p_2-p_1$.
Then the restriction becomes $\Delta \geq 0$.
Further, we set $\eta_0=0.5$ and $\omega=0.25$.
Let  the true values of $p_1$ and $p_2$ be $p_{10,n}=0.5-0.75\Delta_{0,n}$, and $p_{20,n}=0.5+0.25\Delta_{0,n}$, with $\Delta_0=\Delta_{0,n}=\delta  n^{-1/2}$, and $\delta$ being a fixed nonnegative local parameter not depending on $n$.
Under this current setup, $\sigma_0^2=\eta_0(1-\eta_0)=0.5^2$.

Applying Theorem \ref{prop4_bootCI}, we can numerically evaluate the local asymptotic coverage probabilities for $p_1$, $p_2$ and $\Delta$.
The exact coverage probabilities  are calculated
by following the definition of the bootstrap percentile confidence intervals and the properties of binomial random variables.
In Figure~\ref{fig6.coverage.twobinom}, we graph the asymptotic and exact coverage probabilities of the bootstrap percentile confidence intervals for $p_1$, $p_2$ and $\Delta$ versus the true mean difference $\Delta_0$ in the cases of $(n_1,n_2)=(25,75)$ and $(n_1,n_2)=(100,300)$ at nominal level $1-\alpha=0.90$ with $\alpha_1=\alpha_2=0.05$.
For comparison, we also include the exact coverage probabilities of the bootstrap percentile confidence intervals for $p_1$, $p_2$, and $\Delta$
without using the constraint.

\begin{figure} 
\centering
\includegraphics[width=32pc]{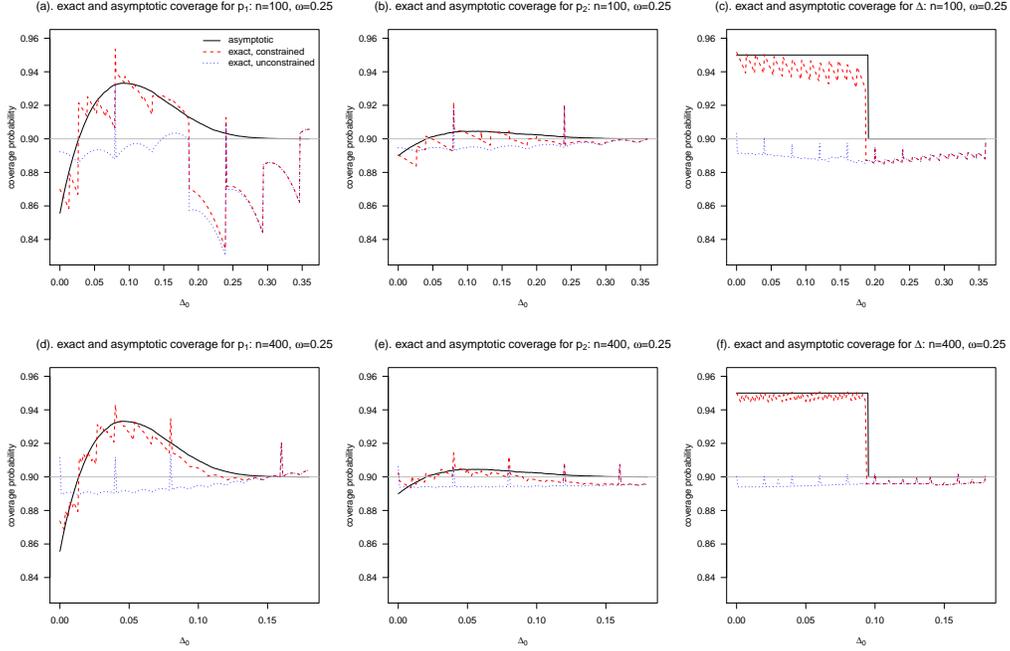}
\caption[]{Exact and asymptotic coverage probabilities of the bootstrap percentile confidence intervals, for the proportions $p_1$ and $p_2$ and their difference $\Delta$ of two binomial distributions with $p_1\leq p_2 $, as a function of $\Delta_0$.
The lines are for quantified local asymptotic coverage probabilities (solid), exact finite sample coverage probabilities (dash) and exact finite sample coverage probabilities without using constraint (dot).
For $\omega=0.25$, top Panels (a)--(c) are for $n=100$, and bottom Panels (d)--(f) are for $n=400$.
For all plots, we set $1-\alpha=0.90$ with $\alpha_1= \alpha_2=0.05$.
}
\label{fig6.coverage.twobinom}
\end{figure}

As we can observe from Figure~\ref{fig6.coverage.twobinom},
the quantified local asymptotic coverage probabilities well capture the general trend of the exact coverage probabilities.
The two coverage probabilities become closer to each other as the sample size increases.
\end{example}


\section{Proofs of the results in Section \ref{sec.case1}}
\label{sec3.one-sample}

\subsection{Proof of Lemma~\ref{lem3_bootCI}}
Recall that the parameter space for $\theta$ is $\mathcal{C}_1=\{\theta: \theta\geq d\}$,
which is a closed convex set.
Then by  Proposition 2.4.3 in \citet[p.~51]{Silva2004},
$\hat\theta_n$
equivalently minimizes
$$
(\bar X_n-\theta)^2
$$
subject to the constraint set $\mathcal{C}_1$.
That is
$$
\hat\theta_n=\arg\min_{\theta \in \mathcal{C}_1}(\bar X_n-\theta)^2=\max(\bar X_n,d).
$$
This finishes the proof of Lemma \ref{lem3_bootCI}.




\subsection{Proof of Lemma~\ref{lem4_bootCI}}
%
%
Recall that, under the proposed local asymptotic framework, $\theta_{0,n}=d+n^{-1/2}\tau$ and $\sigma_0^2=b''(\psi_0)$ with $\psi_0=b'^{-1}(d)$.
Applying the central limit theorem for a triangular array gives
\[
\frac{\sum_{i=1}^n (X_{i}-\theta_{0,n})}{\surd \big({n\sigma_n^2}\big)} \to Z ,
\]
in distribution, as $n\to\infty$, where $\sigma_n^2=b''(\psi_n)$ with $\psi_n=b'^{-1}(\theta_{0,n})$
and $Z\sim N(0,1)$.
Further, both $b'^{-1}(\cdot)$ and $b''(\cdot)$ are continuous functions, and $\theta_{0,n}\to d$ as $n\to\infty$.
Then we have $\sigma_n^2\to\sigma_0^2$ as $n\to\infty$.
By Slutsky's theorem, we have
\ba
\label{case3_clt}
{{n}^{1/2}(\bar X_n-\theta_{0,n})}/{\sigma_0} \to Z ,
\ea
in distribution, as $n\to\infty$.
Together with the continuous mapping theorem, it follows that
\bas
{{n}^{1/2}(\hat\theta_n-\theta_{0,n})}/{\sigma_0} &=& {{n}^{1/2} \{\max( \bar X_n,d) - \theta_{0,n}\} }/{\sigma_0} \\
&=& \max\left\{ {{n}^{1/2}(\bar X_n-\theta_{0,n})}/{\sigma_0}, {{n}^{1/2}(d-\theta_{0,n})}/{\sigma_0} \right\} \\
&\to& \max\left\{Z, -{\tau}/{\sigma_0} \right\},
\eas
in distribution, as $n\to\infty$.
This completes the proof of Lemma \ref{lem4_bootCI}.




\subsection{Proof of Theorem \ref{prop3_bootCI}}
We first recall some notation.
Let $X_1^*,\ldots,X_n^*$
be the bootstrap sample from $f(x;\hat\theta_n)$ for the given $\hat\theta_n$,
and
$
\hat\theta_n^*=\max(\bar X_n^*,d)
$
be the maximum likelihood estimator of $\theta$ based on the bootstrap sample, where
$
\bar X_n^*=\sum_{i=1}^nX_i^*/n.
$
Denote the bootstrap distributions of $\hat\theta_n$
and $\bar X_n$, respectively, by
$$
G_n^*(x;\hat\theta_n)=\Ps\left(\hat\theta_n^*\leq x \mid \hat\theta_n\right),
\quad
\bar G_n^*(x;\hat\theta_n)=\Ps\left(\bar X_n^*\leq x \mid \hat\theta_n\right).
$$
Then
$$
G_n^*(x;\hat\theta_n)
=\left\{
\begin{array}{cc}
\bar G_n^*(x;\hat\theta_n), \quad&x\geq d\\
0,\quad &x<d\\
\end{array}
\right. .
$$
Recall that $q_{\alpha}^*$ is the $\alpha^{\rm th}$ quantile of the bootstrap distribution of $\hat\theta_n$.
It follows that
$$
q_{\alpha}^*=G^{*-1}_n(\alpha;\hat\theta_n)=\max\{\bar G^{*-1}_n(\alpha;\hat\theta_n),d\}.
$$
Therefore
\begin{eqnarray*}
\Pr\left(\theta_{0,n}\in[q^*_{\alpha_1},q^*_{1-\alpha_2}]\right)
&=&
\Pr\left[\max\{\bar G^{*-1}_n(\alpha_1;\hat\theta_n),d\}
\leq \theta_{0,n}\leq\max\{\bar G^{*-1}_n(1-\alpha_2;\hat\theta_n),d\}\right]\\
&=&
\Pr\left\{\bar G^{*-1}_n(\alpha_1;\hat\theta_n)
\leq \theta_{0,n}\leq \bar G^{*-1}_n(1-\alpha_2;\hat\theta_n)\right\}.
\eas
Let \[
\bar H_n^*(x; \hat\theta_n)=\Ps\left\{ {{n}^{1/2}(\bar X_n^*-\hat\theta_n)}/{\sigma_0}\leq x \mid \hat\theta_n\right\},
\]
which is the bootstrap distribution of the standardized $\bar X_n$.
Then
$$
\bar G^{*-1}_n(\alpha;\hat\theta_n)=n^{-1/2}\sigma_0\bar H_n^{*-1}(\alpha; \hat\theta_n)+\hat\theta_n.
$$
Therefore
\begin{eqnarray}
\label{proof.thm1.part0}
\Pr\left(\theta_{0,n}\in[q^*_{\alpha_1},q^*_{1-\alpha_2}]\right)
=\Pr\left\{\bar H_n^{*-1}(\alpha_1;\hat\theta_n)
\leq {{n}^{1/2}(\theta_{0,n}-\hat\theta_n)}/{\sigma_0}\leq \bar H_n^{*-1}(1-\alpha_2;\hat\theta_n)\right\}.\quad
\end{eqnarray}

We next study the asymptotic property of $\bar H_n^{*-1}(\alpha_1;\hat\theta_n)$ in the following lemma, which is very helpful in our proofs.


\begin{lemma}
\label{localquantile}
Under the same setup and assumptions as in Theorem \ref{prop3_bootCI}, we have
\begin{itemize}
\item[](a)~
$\hat\theta_n=d+o_p(1)$ and $\hat\sigma_n^2=\sigma_0^2+o_p(1)$, where
$\hat\sigma_n^2=b''(\hat\psi_n)$
with $\hat\psi_n=b'^{-1}(\hat\theta_{n})$;
\item[](b)~
$\sup_{x}|\bar H_n^*(x; \hat\theta_n)-\Phi(x)|=o_p(1)$;
\item[](c)~
$\bar H_n^{*-1}(\alpha; \hat\theta_n) =  \Phi^{-1}(\alpha)+o_p(1)$
for any given level $\alpha\in(0, 1)$.
\end{itemize}
\end{lemma}
\begin{proof}
We first consider Part (a).
Note that (\ref{case3_clt})
implies that $\bar X_n-\theta_{0,n}=o_p(1)$.
Recall that $\theta_{0,n}=d+n^{-1/2}\tau$.
Then
$$
\bar X_n=d+o_p(1).
$$
This implies that
$$
\hat\theta_n=\max(\bar X_n,d)=d+o_p(1).
$$
Recall that both $b'^{-1}(\cdot)$ and $b''(\cdot)$ are continuous functions.
By the continuous mapping theorem,
we further have
$$
\hat\sigma_n^2=\sigma_0^2+o_p(1).
$$
This finishes the proof of Part (a).

Next we consider Part (b).
We start with finding the limiting distribution of
 ${{n}^{1/2}(\bar X_n^*-\hat\theta_n)}/{\sigma_0}$ for given $\hat\theta_n$.
 Note that 
$$
 E(X_i^*\mid\hat\theta_n)=\hat\theta_n, \quad   \var(X_i^*\mid\hat\theta_n)=\hat\sigma_n^2.
$$
 Then, by Berry-Esse\'{e}n inequaltiy \citep[Section 3.1, p.~74]{Shao1995}
 or the central limit theorem \citep[Theorem~23.4]{van1998},
we have
$$
\sup_{x\in \mathbb{R}} \left| \Pr\left({{n}^{1/2}(\bar X_n^*-\hat\theta_n)}/{\hat\sigma_n}\leq x\mid \hat\theta_n\right)
-\Phi(x)\right| \to 0,
$$
 in probability.
 Recall that in Part (a), we have shown $\hat\sigma_n\to\sigma_0$ in probability.
 By conditional Slutsky's theorem \citep{Cheng2015},
 we further have
$$
\sup_{x\in \mathbb{R}} \left| \Pr\left({{n}^{1/2}(\bar X_n^*-\hat\theta_n)}/{\sigma_0}\leq x\mid \hat\theta_n\right)
-\Phi(x)\right| \to 0,
$$
in probability, which implies that
$$
\sup_{x\in\mathbb{R}} \left|\bar H_n^*(x; \hat\theta_n)-\Phi(x)\right|=o_p(1).
$$
This finishes the proof of Part (b).
\smallskip

With Part (b), then Part (c) is a direct application of Lemma 21.2 in \cite{van1998}. 
\end{proof}


We now move back to our proof of Theorem \ref{prop3_bootCI}.
Applying Lemma \ref{localquantile}
to (\ref{proof.thm1.part0}) gives
\ba
&&\nonumber\Pr\left(\theta_{0,n}\in[q^*_{\alpha_1},q^*_{1-\alpha_2}]\right) \\
\nonumber&=&
\Pr\left\{-\Phi^{-1}(1-\alpha_2)+o_p(1) \leq {{n}^{1/2}(\hat\theta_n-\theta_{0,n})}/{\sigma_0} \leq -\Phi^{-1}(\alpha_1)+o_p(1) \right\} \\
\nonumber&=&\Pr\left\{\Phi^{-1}(\alpha_2)+o_p(1) \leq {{n}^{1/2}(\hat\theta_n-\theta_{0,n})}/{\sigma_0} \leq \Phi^{-1}(1-\alpha_1)+o_p(1) \right\} \\
\nonumber&=&
\Pr\left\{{{n}^{1/2}(\hat\theta_n-\theta_{0,n})}/{\sigma_0} \leq \Phi^{-1}(1-\alpha_1)+o_p(1) \right\} \\
\label{proof.thm1.part1}&& - \Pr\left\{{{n}^{1/2}(\hat\theta_n-\theta_{0,n})}/{\sigma_0} <\Phi^{-1}(\alpha_2)+o_p(1) \right\}.
\end{eqnarray}
Recall that in Lemma \ref{lem4_bootCI}, we have shown that
$$
{{n}^{1/2}(\hat\theta_n-\theta_{0,n})}/{\sigma_0} \to \max\left(Z,-\tau/\sigma_0\right),
$$
in distribution, as $n\to\infty$.
That is, the limiting distribution of ${{n}^{1/2}(\hat\theta_n-\theta_{0,n})}/{\sigma_0}$
is $\Phi(x)I(x\geq-\tau/\sigma_0)$, which is continuous at $x=\Phi^{-1}(1-\alpha_1)$
and
$x=\Phi^{-1}(\alpha_2)$ if $\Phi^{-1}(\alpha_2)\neq-\tau/\sigma_0$.
By the definition of convergence in distribution, Slusky's theorem,
and (\ref{proof.thm1.part1}),
we have that, if $\Phi^{-1}(\alpha_2)\neq-\tau/\sigma_0$, then
$$
\lim_{n\to\infty}
\Pr\left(\theta_{0,n}\in[q^*_{\alpha_1},q^*_{1-\alpha_2}]\right)
=1-\alpha_1-\alpha_2I\left(\Phi^{-1}(\alpha_2)\geq -\tau/\sigma_0\right).
$$
That is,  for every continuous point of the limiting function, we have
\bas
\lim_{n\to\infty}
\Pr\left(\theta_{0,n}\in[q^*_{\alpha_1},q^*_{1-\alpha_2}]\right)
&=&\left\{
\begin{array}{ll}
1-\alpha_1-\alpha_2, \quad& \Phi^{-1}(\alpha_2)> -\tau/\sigma_0\\
1-\alpha_1, \quad& \Phi^{-1}(\alpha_2)< -\tau/\sigma_0
\end{array}
\right.\\
&=&
\left\{
\begin{array}{ll}
1-\alpha_1-\alpha_2, \quad & \tau/\sigma_0>\Phi^{-1}(1-\alpha_2)\\
1-\alpha_1, \quad & \tau/\sigma_0<\Phi^{-1}(1-\alpha_2)
\end{array}
\right..
\eas
This finish the proof of Theorem \ref{prop3_bootCI}.

\section{Proofs of the results in Section \ref{sec.case2}}
\label{sec4.two-sample}

\subsection{Proof of Lemma \ref{lem7_boot}}

Note that  the parameter space $\mathcal{C}_2=\{(\theta_1,\theta_2):\theta_1 \leq \theta_2\}$
is a closed convex set.
Then by Proposition 2.4.3 in \citet[p.~51]{Silva2004},
$(\hat\theta_{n1},\hat\theta_{n2})$
equivalently minimizes
$$
n_1(\bar X_{n1}-\theta_1)^2 +n_2(\bar X_{n2}-\theta_2)^2.
$$
That is,
\be
\label{isotonic.opt}
(\hat\theta_{n1},\hat\theta_{n2})=\arg\min_{(\theta_1,\theta_2)\in \mathcal{C}_2} \left\{ n_1(\bar X_{n1}-\theta_1)^2 +n_2(\bar X_{n2}-\theta_2)^2 \right\}.
\ee
To identify the forms of $(\hat\theta_{n1},\hat\theta_{n2})$,
we consider the following reparameterization
\[
\theta_1 = \eta - (1-\omega)\Delta, \quad \theta_2 = \eta + \omega\Delta,
\]
or equivalently
$$
\eta=\omega\theta_1+(1-\omega)\theta_2, \quad \Delta=\theta_2-\theta_1.
$$
Under the above reparameterization, the constraint $\theta_2\geq\theta_1$ then becomes $\Delta \geq 0$, while $\eta$ is free of restriction.
For the optimization problem (\ref{isotonic.opt}),
the maximum likelihood estimator of $(\eta,\Delta)$ is
$$
(\hat\eta_{n},\hat\Delta_{n})
=\arg\min_{(\eta,\Delta)}\left[n_1\left\{\bar X_{n1}-\eta +(1-\omega)\Delta\right\}^2+n_2(\bar X_{n2}-\eta - \omega\Delta)^2\right]
$$
subject to the constraint $\Delta\geq0$.
After some algebra, we find
$$
(\hat\eta_{n},\hat\Delta_{n})
=\arg\min_{(\eta,\Delta)}\left[n\left\{\eta-\omega\bar X_{n1}-(1-\omega)\bar X_{n2} \right\}^2+n\omega(1-\omega)\{\Delta-(\bar X_{n2}-\bar X_{n1})\}^2\right]
$$
subject to the constraint $\Delta\geq0$.
Hence,
$$
\hat\eta_{n}=\omega\bar X_{n1}+(1-\omega)\bar X_{n2},
\quad
\hat\Delta_n=(\bar X_{n2}-\bar X_{n1})^+,
$$
which implies that
$$\hat\theta_{n1}=\min\left\{\bar X_{n1},\omega \bar X_{n1}+(1-\omega)\bar X_{n2} \right\},
\quad
\hat\theta_{n2}=\max\left\{\bar X_{n2},\omega \bar X_{n1}+(1-\omega)\bar X_{n2} \right\}.
$$
This finishes the proof of Lemma \ref{lem7_boot}.




\subsection{Proof of Lemma \ref{lem8_boot}}
For the convenience of presentation, we introduce some compact notation.
Write $\theta_{0,n}=(\theta_{10,n},\theta_{20,n})^\top$, $\bar X_n=(\bar X_{n1},\bar X_{n2}) ^\top$ and $\hat\theta_{n}=(\hat\theta_{n1},\hat\theta_{n2})^\top$.
Let
\[
\bU_n=(U_{n1}, U_{n2})^\top=\frac{{n}^{1/2}(\bar X_n-\theta_{0,n})}{\sigma_0}.
\]
Write  $W_n=\omega U_{n1}+(1-\omega)U_{n2}$.
Recall that
$\theta_{10,n}=\eta_0-(1-\omega)\Delta_{0,n}$ and $\theta_{20,n}=\eta_0+\omega\Delta_{0,n}$
with $\Delta_{0,n}=\delta  n^{-1/2}$,  and $\sigma_{0}^2=b''(\psi_{0})$ with $\psi_{0}=b'^{-1}(\eta_0)$.

Applying the central limit theorem for a triangular array,
we have
\[
\left(
\begin{array}{cc}
\frac{\sum_{j=1}^{n_1} (X_{1j}-\theta_{10,n})}{\surd{(n_1\sigma_{1,n}^2})}\\
\frac{\sum_{j=1}^{n_2} (X_{2j}-\theta_{20,n})}{\surd{(n_2\sigma_{2,n}^2})}\\
\end{array}
\right)
\to
\left(
\begin{array}{cc}
Z_1\\
Z_2\\
\end{array}
\right) ,
\]
in distribution, as $n\to\infty$, where $\sigma_{i,n}^2=b''(\psi_{i,n})$ with $\psi_{i,n}=b'^{-1}(\theta_{i0,n})$.
Further, we have $\sigma_{i,n}^2\to\sigma_0^2$ as $n\to\infty$,
since $\theta_{i0,n}\to\eta_0$ as $n\to\infty$
and both $b'^{-1}(\cdot)$ and $b''(\cdot)$ are continuous functions.
Therefore, by Slutsky's theorem, we have
\[
\bU_n
\to
\left(
\begin{array}{cc}
{\omega}^{-1/2}Z_1\\
{(1-\omega)}^{-1/2}Z_2\\
\end{array}
\right) ,
\]
in distribution, as $n\to \infty$.

We now come to the limiting distribution of
 ${n^{1/2}(\hat\theta_{n1}-\theta_{10,n})}/{\sigma_0}$.
Using the form of $\hat\theta_{n1}$, we have
\bas
&&{n^{1/2}(\hat\theta_{n1}-\theta_{10,n})}/{\sigma_0} \\
&=& n^{1/2} \left\{ {\min\left\{\bar X_{n1},\omega \bar X_{n1}+(1-\omega)\bar X_{n2} \right\}-\theta_{10,n}}/{\sigma_0} \right\} \\
&=& \min \left\{ {n^{1/2} \left(\bar X_{n1}-\theta_{10,n}\right)}/{\sigma_0}, {n^{1/2}\left(\omega \bar X_{n1}+(1-\omega)\bar X_{n2} -\theta_{10,n}\right)}/{\sigma_0} \right\} \\
&=& \min \left\{ U_{n1}, W_n+ {n^{1/2}\left(\eta_0 - \theta_{10,n}\right)}/{\sigma_0} \right\} \\
&\to& \min\left\{ {\omega}^{-1/2}Z_1, {\omega}^{1/2}Z_1+(1-\omega)^{1/2}Z_2 + {(1-\omega)\delta}/{\sigma_0} \right\}
\eas
in distribution, as $n\to\infty$.

The proof for ${n^{1/2}(\hat\theta_{n2}-\theta_{20,n})}/{\sigma_0}$ is similar, and hence is omitted.
For $\hat\Delta_n$, the proof is similar to that of Lemma \ref{lem4_bootCI}, and hence is also omitted.
This finishes the proof of Lemma \ref{lem8_boot}.




\subsection{Proof of Theorem \ref{prop4_bootCI}}
Next we concentrate on the proof of Part (a) of Theorem \ref{prop4_bootCI} as the proof of Part (b) is just parallel.
Also, the proof of Part (c) is similar to that of Theorem \ref{prop3_bootCI}, and hence is omitted.

We first recall some notation.
Recall that $X_{11}^*,\ldots,X_{1n_1}^*$
is the bootstrap sample from $f(x;\hat\theta_{n1})$ for given $\hat\theta_{n1}$,
and
$X_{21}^*,\ldots,X_{2n_2}^*$
is the bootstrap sample from $f(x;\hat\theta_{n2})$ for given $\hat\theta_{n2}$.
Further, recall that
$(\hat\theta_{n1}^*,\hat\theta_{n2}^*)$
is the maximum likelihood estimator of $(\theta_1,\theta_2)$
based on $X_{ij}^*$'s.
Then
$$\hat\theta_{n1}^*=\min\left\{\bar X_{n1}^*,\omega \bar X_{n1}^*+(1-\omega)\bar X_{n2}^* \right\},
\quad
\hat\theta_{n2}^*=\max\left\{\bar X_{n2}^*,\omega \bar X_{n1}^*+(1-\omega)\bar X_{n2} ^*\right\},
$$
where $\bar X_{ni}^*=\sum_{j=1}^{n_i}X_{ij}^*/n_i$, $i=1,2$.
Recall that the bootstrap distribution of $\hat\theta_{n1}$ is denoted
by
$$
G_{n1}^*(x;\hat\theta_n)
=\Ps(\hat\theta_{n1}^*\leq x \mid \hat\theta_{n}),
$$
and the corresponding $\alpha$-quantile is denoted  by $q_{1,\alpha}^*$.

Next, we mainly consider $\Pr\left(\theta_{10,n}\geq q_{\alpha_1}^*\right)$ in Part (a).
The other side can be similarly proved.
Note that
\ba
\label{thm2.proof.partb.part0}
\Pr\left(\theta_{10,n}\geq q_{\alpha_1}^*\right)
&=&
\Pr\left\{\alpha_1\leq  G_{n1}^*(\theta_{10,n};\hat\theta_n)\right\}.
\ea
For $G_{n1}^*(\theta_{10,n};\hat\theta_{n})$,
we have
\bas
G_{n1}^*(\theta_{10,n};\hat\theta_n)
&=&
\Ps\left[\min\left\{\bar X_{n1}^*,\omega \bar X_{n1}^*+(1-\omega)\bar X_{n2}^* \right\} \leq\theta_{10,n} \mid \hat\theta_n \right] \\
&=&
\Ps\left( \bar X_{n1}^* \leq\theta_{10,n} \mid \hat\theta_n \right)
+\Ps\left( \omega \bar X_{n1}^*+(1-\omega)\bar X_{n2}^* \leq\theta_{10,n} \mid \hat\theta_n \right) \\
&&-\Ps\left\{ \bar X_{n1}^* \leq\theta_{10,n}, \omega\bar X_{n1}^*+(1-\omega)\bar X_{n2}^* \leq\theta_{10,n} \mid \hat\theta_n \right\}.
\eas
For $ i=1,2$, let
$$
U_{ni}^*={n^{1/2} \left(\bar X_{ni}^*-\hat\theta_{ni}\right)}/{\sigma_0},
$$
and
$
W_{n}^*=\omega U_{n1}^*+(1-\omega)U_{n2}^*.
$
After some algebra, we have
\ba
\nonumber
G_{n1}^*(\theta_{10,n};\hat\theta_n)
&=&
\Ps\left(U_{n1}^*\leq -U_{n1} \mid \hat\theta_n \right)
+ \Ps\left\{ W_n^* \leq -W_n-(1-\omega)\delta/\sigma_0 \mid \hat\theta_n\right\} \\
\label{thm2.proof.part0}&&
-\Ps\left\{U_{n1}^*\leq -U_{n1}, W_n^* \leq -W_n-(1-\omega)\delta/\sigma_0 \mid \hat\theta_n \right\}.
\ea

In next lemma, we study the asymptotic properties of
$$
\Ps\left(U_{n1}^*\leq x \mid \hat\theta_n \right),\quad
\Ps\left( W_n^* \leq x \mid \hat\theta_n\right),\quad
\Ps\left(U_{n1}^*\leq x,
W_n^* \leq y \mid \hat\theta_n\right),
$$
which are very helpful in our proofs.

\begin{lemma}
\label{localquantile2}
Under the same setup and assumptions as in Theorem \ref{prop4_bootCI}, we have
\ba
\label{local2.eq1}
&&
\sup_{x\in\mathbb{R}}\left|\Ps\left({\omega}^{1/2}U_{n1}^*\leq x \mid \hat\theta_n \right)-\Phi(x)\right|=o_p(1),\\
\label{local2.eq2}
&&
\sup_{x\in\mathbb{R}}\left|\Ps\left( W_n^* \leq x \mid \hat\theta_n\right)-\Phi(x)\right|=o_p(1),\\
\label{local2.eq3}
&&
\sup_{(x,y)\in \mathbb{R}^2}\left|\Ps\left({\omega}^{1/2}U_{n1}^*\leq x, W_n^* \leq y \mid \hat\theta_n\right)-\bPhi_{(\szero,\bsLambda_1)}(x,y)\right|=o_p(1).
\ea
\end{lemma}
\begin{proof}
Similar to the proof as in Part (b) of Lemma \ref{localquantile},
we can show that
$$
\sup_{x\in \mathbb{R}} \left| \Pr\left({\omega}^{1/2}U_{n1}^*\leq x \mid\hat\theta_n\right)
-\Phi(x)\right| \to 0,
$$
in probability, and
$$
\sup_{x\in \mathbb{R}} \left| \Pr\left((1-\omega)^{1/2}U_{n2}^*\leq x\mid\hat\theta_n\right)
-\Phi(x)\right| \to 0,
$$
in probability.
Further, conditional on $\hat\theta_n$,
${\omega}^{1/2}U_{n1}^*$
and $(1-\omega)^{1/2}U_{n2}^*$
are independent.
Hence,  we have
$$
\sup_{(x,y)\in \mathbb{R}^2} \left| \Pr\left({\omega}^{1/2}U_{n1}^*\leq x, (1-\omega)^{1/2}U_{n2}^*\leq y \mid\hat\theta_n\right)
-\Phi_{(\szero,\bI_{2\times 2})}(x,y)\right| \to 0,
$$
in probability, where $\bI_{2\times 2}$ is a two-by-two identity matrix.
Then by Example 3.3 of \cite{Shao1995}, we have,  that
$$
\sup_{(x,y)\in \mathbb{R}^2} \left| \Pr\left({\omega}^{1/2}U_{n1}^*\leq x, W_n^*\leq y \mid\hat\theta_n\right)
-\Phi_{(\szero,\bsLambda_1)}(x,y)\right| \to 0,
$$
in probability.
This implies (\ref{local2.eq1})--(\ref{local2.eq3}) and finishes the proof.
\end{proof}

Now we move back to  the proof of Theorem \ref{prop4_bootCI}.
Combining Lemma \ref{localquantile2} and the form of $G_{n1}^*(\theta_{10,n};\hat\theta_n)$ in (\ref{thm2.proof.part0}),
we obtain
\ba
\nonumber
G_{n1}^*(\theta_{10,n};\hat\theta_n)
&=&
\Phi(-{\omega}^{1/2}U_{n1})
+ \Phi\left(-W_n-{(1-\omega)\delta}/{\sigma_0}\right) \\
\nonumber&&
-\bPhi_{(\szero,\bsLambda_1)}\left(-{\omega}^{1/2}U_{n1}, -W_n-{(1-\omega)\delta}/{\sigma_0}\right) +o_p(1) \\
\nonumber&=&
\Phi\left\{-C_{11}(U_{n1})\right\} + \Phi\left\{-C_{12}(U_{n1},U_{n2}) \right\}\\
\nonumber&& - \bPhi_{(\szero,\bsLambda_1)}\left\{-C_{11}(U_{n1}), -C_{12}(U_{n1},U_{n2})\right\}+o_p(1)\\
&=&g_1(U_{n1},U_{n2})+o_p(1),
\label{thm2.proof.partb.part2}
\ea
where $g_1(x,y)$ is defined in Part (a) of Theorem \ref{prop4_bootCI}.

Combining (\ref{thm2.proof.partb.part0}) and (\ref{thm2.proof.partb.part2})
and applying Slutsky's theorem,
we further have
\ba
\nonumber
\lim_{n\to\infty}\Pr\left(\theta_{10,n}\geq q_{\alpha_1}^*\right)&=&
\lim_{n\to\infty}
\Pr\left\{
\alpha_1\leq g_1(U_{n1},U_{n2})+o_p(1)
\right\}\\
\label{thm2.proof.partb.part3}
&=&\iint I\{\alpha_1\leq g_1(x,y)\} dF_{12}(x,y).~~~~
\ea
Similarly, we find
\ba
\lim_{n\to\infty}\Pr\left(\theta_{10,n}\leq q_{1-\alpha_2}^*\right)
\label{thm2.proof.partb.part4}
&=&\iint I\{g_1(x,y)\leq 1-\alpha_2\} dF_{12}(x,y).
\ea
Combining (\ref{thm2.proof.partb.part3}) and (\ref{thm2.proof.partb.part4}),
we get
\bas
\nonumber
\lim_{n\to\infty}\Pr\left(q_{\alpha_1}^*\leq \theta_{10,n}\leq q_{1-\alpha_2}^*\right)&=&
\iint I\{\alpha_1\leq g_1(x,y)\leq 1-\alpha_2\} dF_{12}(x,y).
\eas
This finishes the proof of Theorem \ref{prop4_bootCI}.

\section*{Acknowledgements}
This research was partially supported by the Natural Sciences and Engineering Research Council of Canada.

%
%
%
%
%
%
%

\bibliographystyle{imsart-nameyear} 
\bibliography{bootCI}

\end{document}